\documentclass{article}
\usepackage{graphicx} 

\usepackage{graphicx}%
\usepackage{multirow}%
\usepackage{amsmath,amssymb,amsfonts}%
\usepackage{amsthm}%
\usepackage{mathrsfs}%
\usepackage[title]{appendix}%
\usepackage{xcolor}%
\usepackage{textcomp}%
\usepackage{manyfoot}%
\usepackage{booktabs}%
\usepackage{algorithm}%
\usepackage{algorithmicx}%
\usepackage{algpseudocode}%
\usepackage{listings}%

\usepackage [latin1]{inputenc}

\usepackage{amsmath}
\usepackage{cleveref}
\usepackage{amssymb}
 \usepackage{graphicx}
 \usepackage{pgfplots}
\usepgfplotslibrary{groupplots}
\pgfplotsset{compat=1.12}
\usepackage{pgf,tikz}
\usepackage{mathrsfs}
\usetikzlibrary{arrows}
\usetikzlibrary{arrows,shapes,positioning}
\usepackage{caption}
\usepackage{tikz-cd}
\newcommand{\gennum}[1]{\left<#1\right>}
\newcommand{\Use}{\mathcal{\Bar{U}}(\l_t)}
\newcommand{\wtp}{\widetilde{\pi}_n}
\newcommand{\Uti}{\widetilde{\mathcal{U}}}
\newcommand{\va}[1]{\left|#1\right|}
\let \l \lambda
\newcommand{\mup}[1]{\widetilde{\mathcal{U}}_{#1}}

\newtheorem{lemma}{Lemma}
\newtheorem{corollary}{Corollary}
\newtheorem{theorem}{Theorem}
\newtheorem{example}{Example}
\newtheorem{remark}{Remark}
\newtheorem{definition}{Definition}
\newtheorem{proposition}{Proposition}

\title{Unrefinable Partitions into Distinct Parts and Numerical Semigroups}
\author{Lorenzo Campioni\\
lorenzo.campioni1@graduate.univaq.it\\
Disim, Universit\`a degli Studi dell'Aquila, Italy}
\date{}

\begin{document}

\maketitle

\abstract{This article investigates structural connections between unrefinable partitions into distinct parts and numerical semigroups. By analysing the hooksets of Young diagrams associated with numerical sets, new criteria for recognising unrefinable partitions are established.\\
A correspondence between missing parts and the gaps of numerical semigroups is developed, extending previous classifications and enabling the characterisation of partitions with maximal numbers of missing parts. In particular, the results show that certain families of unrefinable partitions correspond precisely to symmetric numerical semigroups when the maximal part is prime. Further structural consequences, examples, and a decomposition of unrefinable partitions by minimal excludant are discussed, together with implications for the study of maximal unrefinable partitions.}\\

\textbf{keywords}:Integer partitions into distinct parts, Numerical semigroups, Combinatorics.

\section{Introduction}\label{sec1}

Unrefinable partitions into distinct parts were introduced in \cite{ACGS3}, where the authors established a connection with the generators of a chain of normalizers. This chain of normalizers was initially defined in \cite{ACGS1} starting from the normalizer $N_n^0$ in $Sym(2^n)$ of the translation group on $\mathbb{F}_2^n$ and subsequently defining the $i$-th term $N^i_n$ as $N_{Sym(2^n)}(N_n^{i-1})$. In \cite{ACGS2} the authors proved that the quantity $\log_2\va{N_n^i\colon N_n^{i-1}}$ is independent of $n$ for $1\leq i\leq n-2$ and in particular it is equal to the $(i+2)$-th term of the sequence of the partial sums of the sequence $\{b_j\}$, where $b_j$ denotes the number of partitions of $j$ into at least two distinct parts. Finally, in \cite{ACGS3}, it was observed that $\log_2\va{N_n^{n-1}\colon N_n^{n-2}}$ is related to the number of unrefinable partitions of $n$ satisfying a condition on their minimal excludant, i.e., the smallest positive integer that does not appear in the partition.\\

 The condition of being unrefinable imposes a non-trivial constraint on the possible distributions of the parts in a partition. In \cite{ACCL1}, Aragona et al. constructed two algorithms: one to determine whether a given sequence is an unrefinable partition and the other to enumerate all unrefinable partitions of a given weight. Later, the same authors in \cite{ACCL2} provided a classification of maximal unrefinable partitions for triangular numbers, and in \cite{ACC}, they extended this classification to the general case. These classifications establish explicit bijections between unrefinable partitions and suitable subsets of partitions into distinct parts.\\
 
In this paper, we prove a connection between unrefinable partitions and numerical semigroups, which are additive submonoids of the non-negative integers that contain $0$ and have finite complement in $\mathbb{N}_0$. This relationship allows us to develop new criteria for recognising unrefinable partitions, by analysing the hooksets of the Young tableaux associated with the corresponding numerical semigroup (see \Cref{condsemi}). Moreover, this connection enables the description of further subsets of unrefinable partitions that fix the largest part and maximise the number of missing parts. In this setting, we identify a relation between these subsets and the set of symmetric numerical semigroups $(SNS)$, particularly when the largest part $\l_t$ is a prime number.\\

The remainder of the paper is organized as follows: in Section \ref{preliminari}, we introduce some basic properties of unrefinable partitions and numerical semigroups. Section \ref{finale} is dedicated to exploring and describing the relationships between them.

\section{Preliminaries}\label{preliminari}

\subsection{Unrefinable Partitions into Distinct Parts}

\begin{definition}
    Let $\l=(\l_1,\ldots,\l_t)$ be a partition into distinct parts, i.e., $\l_1<\l_2<\cdots<\l_t$, with $t\geq2$. The \textbf{set of missing parts} of $\l$, denoted $\mathcal{M}_{\l}=\{\mu_1,\ldots,\mu_m\}$, consists of all integers less than or equal to $\l_t$ that do not appear in $\l$:
    \[
    \mathcal{M}_{\l}\colon=\{1,2,\ldots,\l_t\}\setminus\{\l_1,\ldots,\l_t\}.
    \]
    The smallest missing part is called the \textbf{minimal excludant} of $\l$, written as $\mathrm{mex}(\l)=\mu_1$, with the convention $\mathrm{mex}(\l)=0$ if $\mathcal{M}_{\l}=\emptyset$.
\end{definition}

\begin{definition}
    Let $N \in \mathbb{N}$. A partition $\l=(\l_1,\ldots,\l_t)$ of $N$ into distinct parts, with missing parts $\mu_1 < \mu_2 < \dots <\mu_m$, is said to be \textbf{refinable} if there exist indices $1 \leq \ell \leq t$ and $1\leq i_1<\ldots<i_k \leq m $, with $k\geq 2$, such that $\mu_{i_1}+\cdots+\mu_{i_k} = \l_{\ell}$. Otherwise, $\l$ is \textbf{unrefinable}. The set of unrefinable partitions is denoted by $\mathcal{U}$, and $\mathcal{U}_N$ denotes the subset of partitions of $N$.
\end{definition}

For example, $\l=(1,2,3,5,6,8,9,11,13)$ is refinable since $11 = 4+7$, while $\l=(1,2,3,5,6,8,9,13)$ is unrefinable.  
The condition of being unrefinable imposes non-trivial constraints on the largest part and on the distribution of parts.

\begin{lemma}\label{M=0,1}
    Let $\l=(\l_1,\ldots,\l_t)$ be a partition into distinct parts. If $\#\mathcal{M}_{\l}=\{0,1\}$, then $\l\in\mathcal{U}$.
\end{lemma}

\begin{proof}
    No part of the partition can be expressed as a sum of two or more distinct missing parts, hence $\l$ is unrefinable.
\end{proof}

The converse does not hold in general. For instance, $\l=(1,2,3,5,6,8,9,13)$ is unrefinable with $\#\mathcal{M}_\l=5$.\\

The following definition will be useful in the formulation of the next property and in the derivation of subsequent results.

\begin{definition}\label{Triang}
    For $n \in \mathbb{N}$ the \textbf{n-th triangular number} is 
    \[
    T_n= \sum_{i=1}^n i = \dfrac{n(n+1)}{2}.
    \]
    The \textbf{ complete partition} $\pi_n=(1,2,\ldots,n)$ is the partition of $T_n$ with no missing parts. For $n\geq 3$ and $1\leq d\leq n-1$, let $T_{n,d}=T_n-d$ and define $\pi_{n,d}=(1,\ldots,d-1,d+1,\ldots,n)\vdash T_{n,d}$.
\end{definition}

Note that $T_{n-1} < T_{n,d} < T_n$ for every $n\ge 3$ and $1\le d\le n-1$.

\begin{corollary}
    For every $N\in\mathbb{N}$ such that $N>2$, the set $\mathcal{U}_N$ is nonempty. 
\end{corollary}

\begin{proof}
    If $N=T_n$, then $\pi_n \in \mathcal{U}_N$ by Lemma~\ref{M=0,1}. Otherwise, there exists $1\le d\le n-1$ such that $N=T_{n,d}$, and $\pi_{n,d}\in\mathcal{U}_N$.
\end{proof}

The following result bounds the number of missing parts in an unrefinable partition.

\begin{lemma}\label{lemma:bound}
    Let $\l=(\l_1,\ldots,\l_t)$ be an unrefinable partition and let $\mu_1<\cdots<\mu_m$ be the missing parts. Then the number of missing parts $m$ satisfies the following upper bound:
    \begin{equation}\label{bound1}
            m \leq \left\lfloor \dfrac{\l_t}{2}\right\rfloor.
    \end{equation}
\end{lemma}
\begin{proof}
    Each missing part $\mu_i$ must have a counterpart $\l_t - \mu_i \in \l$, otherwise $\l_t = (\l_t - \mu_i) + \mu_i$ would yield a refinement.  
    Considering the complete partition $\pi_{\l_t}$ and removing as many parts as possible different from $\l_t$ gives the stated bound.
\end{proof}

\begin{proposition}\label{proposition:simplerefinement}
    If $\l$ has a refinement, its smallest refinable part $\l_r$ can be written as $\l_r=a+b$, where $a,b\in\mathcal{M}_{\l}$.
\end{proposition}
\begin{proof}
    Let $\l_r$ be the smallest refinable part and $\l_r = \mu_{i_1} + \cdots + \mu_{i_t}$ a refinement.  
    If $t=2$, the statement holds. Otherwise, set $a = \mu_{i_1}$ and $b = \mu_{i_2}+\cdots+\mu_{i_t}$.  
    If $b \in \l$, then $b < \l_r$ would be refinable, contradicting the minimality of $\l_r$. Hence $b \notin \l$, proving the claim.
\end{proof}

\begin{remark}\label{forbele}
Let $\l$ be an unrefinable partition, and suppose $\mu_1=\mathrm{mex}(\l)$ and $\mu_i\in\mathcal{M}_{\l}$. By definition, the element $\mu_i+\mu_1$ does not belong to $\l$. Analogously, the element $\mu_i+2\mu_1=(\mu_i+\mu_1)+\mu_1\notin\l$ and, by induction, any $x\equiv \mu_i\pmod{\mu_1}$ with $x\geq\mu_i$ is not contained in $\l$.
\end{remark}

To analyse the structure of unrefinable partitions, a vector $\Vec{p}_\l$ is introduced. It records, for each residue class modulo the minimal excludant $\mu_1 = \mathrm{mex}(\l)$, thresholds associated with the arithmetic relations generated by the missing parts. These thresholds allow to check refinability once the vector is constructed.\\
 
Let $\l$ be a partition with missing parts 
$\mathcal{M}_\l = \{\mu_1<\mu_2<\dots<\mu_m\}$.  
The vector 
\[
\Vec{p}_\l = (p_0,p_1,\dots,p_{\mu_1-1})
\]
is constructed iteratively, following the strategy introduced in \cite{ACCL1}. The method is adapted to the present setting by modifying some of the technical steps originally designed for computational efficiency, while preserving the overall structure of the construction. All entries are initially set to $\infty$.

Each missing part $\mu_i$ with $i\ge 2$ produces the arithmetic progression
\[
\mu_1 + k\mu_i, \qquad k\ge 1,
\]
whose terms fall into well-defined residue classes modulo $\mu_1$. Each integer of the progression is considered as a potential contributor to a refinement. If such element appears in $\l$, this provides a refinement in $\l$. Therefore, the progression terms are used to update the corresponding entry of $\Vec{p}_\l$.\\

The residue classes affected by the progression depend on $\gcd(\mu_1,\mu_i)$, according to the following cases:

\begin{itemize}
    \item If $\gcd(\mu_1,\mu_i)=1$, the progression covers all residue classes modulo $\mu_1$, and each term may update the corresponding entry of $\Vec{p}_\l$. Then $1\leq k\leq \mu_1$.
    
    \item If $\gcd(\mu_1,\mu_i)=d\notin\{1,\mu_1\}$, then the progression occupies exactly the residue classes of the subgroup generated by $\mu_i$ in $\mathbb{Z}/\mu_1\mathbb{Z}$, and only those entries $p_j$ may be updated. Then $1\leq k\leq \frac{\mu_1}{d}$.
    
    \item If $\gcd(\mu_1,\mu_i)=\mu_1$, all integers of the form $\mu_1 + k \mu_i$ lie in the zero residue class, so only $p_0$ may be updated.
\end{itemize}

For each subsequent missing part $\mu_i$, let $r \equiv \mu_i \pmod{\mu_1}$. If the current entry $p_r$ is already smaller than $\mu_i$, no update is required. Otherwise, $p_r$ is set to $\mu_i$, and the following steps are performed:

\begin{enumerate}
    \item \textbf{Progression updates.} Update the entries of $\Vec{p}_\lambda$ using all integers of the progression  
    \[
        \mu_1 + k\mu_i, \qquad k \ge 1,
    \]
    according to the same rules as above;

    \item \textbf{Mixed sum updates.} For each entry $p_j$ in a different residue class from $\mu_i$, update the entry corresponding to the residue class of $p_j + \mu_i$ modulo $\mu_1$;

    \item \textbf{Closure under new values.} Each newly updated entry $p_s$ is processed in the same way: for all sums of the form  
    \[
        p_s + p_j
    \]
     with $p_j$ in a different residue class, update the corresponding entries if a smaller value is obtained.
\end{enumerate}

Repeat this iterative procedure for all subsequent missing parts $\mu_i$.

\begin{definition}\label{plambda}
Let $\l$ be a partition into distinct parts. The vector $\Vec{p}_\l$, constructed as above, is called the \textbf{vector of forbidden elements}. For each residue class $0\le i\le \mu_1-1$, the entry $p_i$ is the smallest integer congruent to $i$ modulo $\mu_1$ that, if it were present in $\l$, could participate in a refinement of the partition.
\end{definition}

\begin{remark}
Each entry $p_i$ can be seen as a threshold. A partition $\l$ is unrefinable if and only if all its parts are strictly smaller than the corresponding $p_i$. Any part $\l_\ell \ge p_i$ could participate in a refinement, making $\l$ refinable. Hence, the vector provides a simple criterion to check refinability.
\end{remark}

The following example illustrates the procedure.

\begin{example}
    Consider a partition $\l$ with $\mathrm{mex}(\l)=6$.\\
    The vector $\Vec{p}_\l$ of length $\mu_1 = 6$ is initialized to
    $$
    \begin{tikzpicture}
        \draw (0,0) rectangle(0.5,-0.5) rectangle (1,0) rectangle(1.5,-0.5) rectangle (2,0) rectangle(2.5,-0.5) rectangle (3,0);
        \node at (0.25,-0.25) {$\infty$};
        \node at (0.75,-0.25) {$\infty$};
        \node at (1.25,-0.25) {$\infty$};
        \node at (1.75,-0.25) {$\infty$};
        \node at (2.25,-0.25) {$\infty$};
        \node at (2.75,-0.25) {$\infty$};
        \node at (-0.5,-0.25) {$\Vec{p}_{\l}=$};
    \end{tikzpicture}.
    $$

    Assume $\mu_2=7$, then set $p_1=7$. The arithmetic progression $6 + k \cdot 7$, $1\leq k \le \mu_1$, is computed and the vector is updated according to the residue classes of each term, yielding
    $$
    \begin{tikzpicture}
        \draw (0,0) rectangle(0.5,-0.5) rectangle (1,0) rectangle(1.5,-0.5) rectangle (2,0) rectangle(2.5,-0.5) rectangle (3,0);
        \node at (0.25,-0.25) {$48$};
        \node at (0.75,-0.25) {$7$};
        \node at (1.25,-0.25) {$20$};
        \node at (1.75,-0.25) {$27$};
        \node at (2.25,-0.25) {$34$};
        \node at (2.75,-0.25) {$41$};
        \node at (-0.5,-0.25) {$\Vec{p}_{\l}=$};
    \end{tikzpicture}.
    $$

    Assume $\mu_3=9$. Its residue class modulo $6$ is $3$ and $p_3=27>9$, then $p_3$ is set to $9$. The progression $6 + k \cdot 9$ updates the vector to
    $$
    \begin{tikzpicture}
        \draw (0,0) rectangle(0.5,-0.5) rectangle (1,0) rectangle(1.5,-0.5) rectangle (2,0) rectangle(2.5,-0.5) rectangle (3,0);
        \node at (0.25,-0.25) {$24$};
        \node at (0.75,-0.25) {$7$};
        \node at (1.25,-0.25) {$20$};
        \node at (1.75,-0.25) {$9$};
        \node at (2.25,-0.25) {$34$};
        \node at (2.75,-0.25) {$41$};
        \node at (-0.5,-0.25) {$\Vec{p}_{\l}=$};
    \end{tikzpicture}.
    $$

    Next, the mixed sum is applied: for each $p_j$ in a different residue class from $p_3$, $\mu_3 = 9$ is added and the corresponding entry is updated if smaller, giving
    $$
    \begin{tikzpicture}
        \draw (0,0) rectangle(0.5,-0.5) rectangle (1,0) rectangle(1.5,-0.5) rectangle (2,0) rectangle(2.5,-0.5) rectangle (3,0);
        \node at (0.25,-0.25) {$24$};
        \node at (0.75,-0.25) {$7$};
        \node at (1.25,-0.25) {$20$};
        \node at (1.75,-0.25) {$9$};
        \node at (2.25,-0.25) {$16$};
        \node at (2.75,-0.25) {$29$};
        \node at (-0.5,-0.25) {$\Vec{p}_{\l}=$};
    \end{tikzpicture}.
    $$

    The newly updated entries are processed iteratively, summing them with entries in different residue classes, which produces
    $$
    \begin{tikzpicture}
        \draw (0,0) rectangle(0.5,-0.5) rectangle (1,0) rectangle(1.5,-0.5) rectangle (2,0) rectangle(2.5,-0.5) rectangle (3,0);
        \node at (0.25,-0.25) {$24$};
        \node at (0.75,-0.25) {$7$};
        \node at (1.25,-0.25) {$20$};
        \node at (1.75,-0.25) {$9$};
        \node at (2.25,-0.25) {$16$};
        \node at (2.75,-0.25) {$23$};
        \node at (-0.5,-0.25) {$\Vec{p}_{\l}=$};
    \end{tikzpicture}.
    $$

    Finally, for $\mu_4 = 13$, the entry corresponding to $\mu_4$ modulo $6$ is already $p_1 = 7$, so no update is required.\\
    Consider the partition $\l=(1,2,3,4,5,8,10,11,12,14,15,17,18)$, note that $\mu_1=6$, $\mu_2=7$, $\mu_3=9$ and $\mu_4=13$. It is refinable because the part $15\equiv 3 \pmod{6}$ exceeds $p_3=9$ and can be expressed as a sum of missing parts: $15 = 6 + 9$. \\
    In contrast, the partition $\l = (1,2,3,4,5,8,10,11,12,14,17)$ is unrefinable, since all parts are strictly smaller than the corresponding thresholds in $\Vec{p}_\l$. No part can be written as a sum of missing parts while preserving distinctness.
\end{example}

For completeness, the following cases are also described.

\begin{example}
    Let $\l$ be a partition such that $\mathrm{mex}(\l)=6$.\\
    If $\mu_2=10$, then $\gcd(6,10)=2$. The vector $\Vec{p}_\l$ after processing $\mu_2$ is
$$
    \begin{tikzpicture}
        \draw (0,0) rectangle(0.5,-0.5) rectangle (1,0) rectangle(1.5,-0.5) rectangle (2,0) rectangle(2.5,-0.5) rectangle (3,0);
        \node at (0.25,-0.25) {$36$};
        \node at (0.75,-0.25) {$\infty$};
        \node at (1.25,-0.25) {$26$};
        \node at (1.75,-0.25) {$\infty$};
        \node at (2.25,-0.25) {$10$};
        \node at (2.75,-0.25) {$\infty$};
        \node at (-0.5,-0.25) {$\Vec{p}_{\l}=$};
        \end{tikzpicture}.
$$
If $\mu_2=18$, then $\gcd(6,18)=6$. The vector $\Vec{p}_\l$ becomes
$$
    \begin{tikzpicture}
        \draw (0,0) rectangle(0.5,-0.5) rectangle (1,0) rectangle(1.5,-0.5) rectangle (2,0) rectangle(2.5,-0.5) rectangle (3,0);
        \node at (0.25,-0.25) {$18$};
        \node at (0.75,-0.25) {$\infty$};
        \node at (1.25,-0.25) {$\infty$};
        \node at (1.75,-0.25) {$\infty$};
        \node at (2.25,-0.25) {$\infty$};
        \node at (2.75,-0.25) {$\infty$};
        \node at (-0.5,-0.25) {$\Vec{p}_{\l}=$};
        \end{tikzpicture}.
$$
\end{example}

\begin{remark}
    In the previous example, some entries of the vector $\Vec{p}_\l$ remain equal to $\infty$. This indicates that, in principle, additional parts could be inserted in the partition respecting the corresponding residue classes, without immediately generating a refinement. Such a situation motivates the technical definition of unrefinable sequences of parts.
\end{remark}

\begin{definition}\label{def:unrefinable_sequence}
    Let $\l=(v_1,v_2,\dots)$ be a finite or infinite sequence of distinct positive integers. The sequence $\l$ is called an \textbf{unrefinable sequence of parts} if no element of $\l$ can be expressed as a sum of two or more distinct integers not contained in the sequence. In the finite case, these integers coincide with the usual missing parts of a partition of a number.
\end{definition}

\begin{example}
    Consider the sequence of all odd numbers, $\l=(1,3,5,\dots)$. 
    All missing parts are even, so their sums cannot produce new odd parts. Therefore, $\l$ forms an infinite unrefinable sequence.
\end{example}

\begin{definition}\label{def:saturated}
  Let $\l=(v_1,\ldots,v_l)$ be an unrefinable sequence of parts with
  $\mathrm{mex}(\l)=\mu_1$. The vector $\Vec{p}_{\l}$ is \textbf{saturated} if all its entries are finite, that is,
  \begin{equation*}
   \#\{p_{j} \leq l \;\mid\; 0 \leq j \leq \mu_1-1 \} = \mu_1.
  \end{equation*}
\end{definition}

\begin{lemma}
    Let $\l$ be an unrefinable sequence of parts with $\#\mathcal{M}_{\l}\ge 2$. If there exists another missing part $\mu_i$ such that $\gcd(\mu_1,\mu_i)=1$, then $\Vec{p}_{\l}$ is saturated and $\l$ has finite length. Otherwise, the sequence may continue indefinitely, giving an infinite unrefinable sequence.
\end{lemma}

\begin{proof}
    If $\gcd(\mu_1,\mu_i)=1$, the progression $\mu_1+k\mu_i$ eventually fills all $\mu_1$ residue classes modulo $\mu_1$, and the vector $\Vec{p}_{\l}$ becomes saturated. Hence, no further parts can be added, giving a finite sequence.\\
    If $\gcd(\mu_1,\mu_i)=d\neq1$, only $\mu_1/d$ residue classes are affected, so $\Vec{p}_{\l}$ is not saturated. Consequently, additional parts can be inserted indefinitely, leading to an infinite sequence.
\end{proof}

\subsection{Numerical Semigroups}

This section introduces numerical semigroups, which provide a fundamental tool for analysing unrefinable partitions.

\begin{definition}
    A \textbf{numerical set} $S$ is a subset of the non-negative integers $\mathbb{N}_0$ such that $0\in S$ and the complementary set $S^c=\mathbb{N}_0\setminus S$ is finite. A numerical set $S$ is a \textbf{numerical semigroup} if it is an additive submonoid of $\mathbb{N}_0$, i.e., the sum of any two elements of $S$ also belongs to $S$. The set of all numerical semigroups is denoted by $NS$.\\
    The set $S^c$ is called set of $\textbf{gaps}$ and its cardinality $G(S)$ is called \textbf{genus} of $S$.\\
    The maximal element $F(S)$ of $S^c$ is called \textbf{Frobenius number}. The \textbf{multiplicity} $M(S)$ of $S$ is the smallest non-zero element of $S$.
    \end{definition}

\begin{definition}
    Let $A$ be a nonempty subset of $\mathbb{N}_0$.  
    The submonoid generated by $A$ is
    \[
        \gennum{A}
        = \{ i_1 a_1 + \cdots + i_n a_n 
        \mid n \in \mathbb{N},\ \{a_1,\dots,a_n\} \subseteq A,\ 
        i_1,\dots,i_n \in \mathbb{N}_0 \}.
    \]
    If a numerical semigroup $S$ satisfies $S = \gennum{A}$, then $A$ is called a
    \textbf{system of generators} of $S$.
    
    A system of generators $A$ is \textbf{minimal} if
    $S \neq \gennum{B}$ for every proper subset $B \subset A$.
    Every numerical semigroup admits a unique minimal system of generators
    (see, for instance, \cite{RGS}).
    This unique set is denoted by $msg(S)$, and its cardinality
    $e(S) = \#msg(S)$ is called the \textbf{embedding dimension} of $S$.
\end{definition}

\begin{example}
Let $S=\{0,3,6,8,9,11,12,14,\rightarrow\}$, where $\rightarrow$ indicates that all integers greater than $14$ belong to $S$. Then $S$ is a numerical semigroup whose set of gaps is $S^c=\{1,2,4,5,7,10,13\}$, genus $G(S)=7$, Frobenius number $F(S)=13$ and multiplicity $M(S)=3$. Moreover, $msg(S)=\{3,8,16\}$ and thus $e(S)=3$.
\end{example}

\begin{lemma}
    Let $A$ be a nonempty subset of $\mathbb{N}_0$. Then $\gennum{A}$ is a numerical semigroup if and only if $gcd(A)=1$
\end{lemma}

\begin{proof}
    See \cite[Lemma 2.1]{RGS}.
\end{proof}

\begin{definition}
    Let $S$ be a numerical semigroup and $n\in S\setminus\{0\}$. For $1\leq i\leq n-1$, let $w(i)$ be the smallest element of $S$ such that $w(i)\equiv i \pmod{n}$. The set $Ap(S,n)=\{0,w(1),\ldots,w(n-1)\}$ is called the \textbf{Ap\'ery set} of $S$ with respect to $n$.
\end{definition}

\begin{example}
    Let $S=\{0,3,6,8,9,11,12,14,\rightarrow\}$. Then:
    \begin{itemize}
        \item $Ap(S,3)=\{0,16,8\}$;
        \item $Ap(S,5)=\{0,6,12,8,9\}$.
    \end{itemize}
\end{example}

Keith and Nath showed in \cite{KN} that every numerical set uniquely defines a integer partition. Indeed, given a numerical set 
$$S=\{0,s_1,\ldots,s_n,\rightarrow\}$$
a Young diagram $Y_S$ can be constructed by drawing a contiguous polygonal path from the origin in $\mathbb{Z}^2$ as follows. For each integer $k$ from $0$ to $F(S)$,
\begin{itemize}
    \item draw an \textbf{east} step if $k\in S$,
    \item draw a \textbf{north} step if $k\notin S$
\end{itemize}

Note that the Young diagram $Y_S$ has $G(S)$ rows and $n$ columns, where $S=\{0,s_1,\ldots,s_n,\rightarrow\}$.\\

\begin{example}
    Let $S=\{0,3,6,8,9,11,12,14,\rightarrow\}$. The associated  Young diagram

$$
    
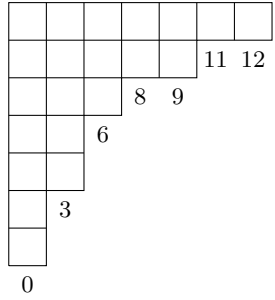
\begin{tikzpicture}
        \draw (0,0)--(0.5,0)--(0.5,1)--(1,1)--(1,2)--(1.5,2)--(1.5,2.5)--(2.5,2.5)--(2.5,3)--(3.5,3)--(3.5,3.5)--(0,3.5)--(0,0);
        \draw (0.5,3.5)--(0.5,1);
        \draw (1,3.5)--(1,2);
        \draw (1.5,3.5)--(1.5,2.5);
        \draw (2,3.5)--(2,2.5);
        \draw (3,3.5)--(3,3);
        \draw (2.5,3.5)--(2.5,3);
        \draw (0,0.5)--(0.5,0.5);
        \draw (0,1)--(0.5,1);
        \draw (0,1.5)--(1,1.5);
        \draw (0,2)--(1,2);
        \draw (0,2.5)--(1.5,2.5);
        \draw (0,3)--(2.5,3);
        \node at (0.25,-0.25) {$0$};
        \node at (0.75,0.75) {$3$};
        \node at (1.25,1.75) {$6$};
        \node at (1.75,2.25) {$8$};
        \node at (2.25,2.25) {$9$};
        \node at (2.75,2.75) {$11$};
        \node at (3.25,2.75) {$12$};
        \end{tikzpicture}
$$
and the corresponding partition is $\l_S=(7,5,3,2,2,1,1)$.
\end{example}
This construction defines a bijection between numerical sets and integer partitions.

\begin{definition}
    Given a cell $c$ in a Young diagram, the \textbf{arm} of $c$ is the number of cells to its right in the same row, and the \textbf{leg} of $c$ is the number of cells below it in the same column. The \textbf{hook} of $c$ is defined as the sum of its arm and leg, plus one. The set of all hook lengths of the diagram is called the \textbf{hookset}.
\end{definition}

The following result, proved in \cite{TKG}, relates the hookset of a Young diagram to the numerical semigroup associated with it.

\begin{proposition}\label{hookset}
    Let $S=\{0,s_1,\ldots,s_n,\rightarrow\}$ be a numerical set with corresponding Young diagram $Y_S$. Then:
    \begin{enumerate}
        \item The hook length of the cell in the first column and $i$-th row is the $i$-th gap of $S$;
        \item For each $0\leq i\leq n-1$ the hook length of the top cell of the $i$-th column of $Y_S$ is $F(S)-s_i$;
        \item S is a numerical semigroup if and only all hook lengths appear in the first column.
    \end{enumerate}
\end{proposition}

\begin{proof}
    Let $1\le j\le n-1$ be the smallest integer such that $s_j\neq j$, so that $s_1^c=j$, where $S^c=\{s_1^c,s_2^c,\ldots,F(S)\}$. Hence, the first row of $Y_S$ contains $j$ east steps, and the hook length of its first cell is $s_1^c=j$. Proceeding by induction, assume that the length of each hook in the first column, from the first row to the $i$-th row, coincides with the gaps of $S$, hence the hook length of the cell in the first column and $i$-th row is $s_i^c$. Let $k$ be such that $s_{k-1}<s_i^c<s_k$. If $s_k\neq s_i^c+1$, then $s_{i+1}^c=s_i^c+1$ and the $(i+1)$-th row of $Y_S$ has the same number of columns of the $i$-th row, so the hook length is $s_i^c+1$. Otherwise, if $s_k=s_i^c+1$, be $j>k$ the lowest integer such that $s_j\neq s_i^c+j+1-k$. Then $j-k$ integers lie between $s_i^c$ and $s_{i+1}^c=s_i^c+j+1-k$, or in other words the $(i+1)$-th row of $Y_S$ has $j-k$ columns more than the $i$-th row. This proves item~$(1)$.\\    
    Now, the upper left corner hook coincides with $F(S)=F(S)-s_0$. Assume that the hook in the $i$-th column is $F(S)-s_i$. Let $k=s_{i+1}-s_i$. After the $i$-th column, the path in $Y_S$ contains $k-1$ north steps, then the hook of the $(i+1)$-th column is equal to $F(S)-s_i-(k-1)-1=F(S)-k=F(S)-s_{i+1}$, establishing item~$(2)$.\\
    Item~$(3)$ remains to check. Let $S=\{s_0,\ldots,s_n,\rightarrow\}$ be a numerical semigroup. Notice that every $\widehat{S}$ such that $\widehat{S}^c=\{s_1^c,\ldots,s_i^c\}$, where $s_i^c<F(S)$, is a numerical semigroup. Furthermore, by item $(2)$, the hook of the $i$-th row and $j$-th column in $Y_S$ is equal to $s_i^c-s_j$. Suppose that exists a cell with hook length $s_i^c-s_j$ that doesn't appear in the first column, so this number is not in $S^c$.  Then $s_h=s_i^c-s_j\in S$, but $s_h+s_j=s_i^c\in S$, a contradiction. Conversely, if all the cells in the $i$-th row are signed by numbers appearing in the first column, then, by item~$(1)$, they are elements of $S^c$. Thus, no two elements of $S$ sum to $s_i^c$, and it follows that $S$ is a numerical semigroup.
\end{proof}

\begin{example}
Consider the numerical semigroup $S=\{0,3,6,8,9,11,12,14,\rightarrow~\}$. The hook lengths in the Young diagram are:

$$
    \begin{tikzpicture}
        \draw (0,0)--(0.5,0)--(0.5,1)--(1,1)--(1,2)--(1.5,2)--(1.5,2.5)--(2.5,2.5)--(2.5,3)--(3.5,3)--(3.5,3.5)--(0,3.5)--(0,0);
        \draw (0.5,3.5)--(0.5,1);
        \draw (1,3.5)--(1,2);
        \draw (1.5,3.5)--(1.5,2.5);
        \draw (2,3.5)--(2,2.5);
        \draw (3,3.5)--(3,3);
        \draw (2.5,3.5)--(2.5,3);
        \draw (0,0.5)--(0.5,0.5);
        \draw (0,1)--(0.5,1);
        \draw (0,1.5)--(1,1.5);
        \draw (0,2)--(1,2);
        \draw (0,2.5)--(1.5,2.5);
        \draw (0,3)--(2.5,3);
        \node at (0.25,0.25) {$1$};
        \node at (0.25,0.75) {$2$};
        \node at (0.25,1.25) {$4$};
        \node at (0.25,1.75) {$5$};
        \node at (0.25,2.25) {$7$};
        \node at (0.25,2.75) {$10$};
        \node at (0.25,3.25) {$13$};

        \node at (0.75,1.25) {$1$};
        \node at (0.75,1.75) {$2$};
        \node at (0.75,2.25) {$4$};
        \node at (0.75,2.75) {$7$};
        \node at (0.75,3.25) {$10$};

        \node at (1.25,2.25) {$1$};
        \node at (1.25,2.75) {$4$};
        \node at (1.25,3.25) {$7$};
        
        \node at (1.75,2.75) {$2$};
        \node at (1.75,3.25) {$5$};
        \node at (2.25,2.75) {$1$};
        \node at (2.25,3.25) {$4$};
        \node at (2.75,3.25) {$2$};
        \node at (3.25,3.25) {$1$};
        \end{tikzpicture}
$$
\end{example}

\begin{definition}\label{SNS}
    A numerical semigroup $S$ is called \textbf{symmetric} if $F(S)$ is odd and if $x\in S^c$ implies $F(S)-x\in S$. The family of symmetric numerical semigroups is denoted by $SNS$.\\
    A numerical semigroup is \textbf{pseudo-symmetric} if $F(S)$ is even and $x\in S^c$ implies $F(S)-x\in S$ or $x=\frac{F(S)}{2}$.
\end{definition}

\begin{lemma}
    Let $S$ be a numerical semigroup. Then:
    \begin{itemize}
        \item $S$ is symmetric if and only if $G(S)=\frac{F(S)+1}{2}$;
        \item $S$ is pseudo-symmetric if and only if $G(S)=\frac{F(S)+2}{2}$.
    \end{itemize} 
\end{lemma}
 \begin{proof}
     See \cite{RGS} Corollary $4.5$.
     \end{proof}

\section{Connection between Unrefinable Partitions and Numerical Semigroups}\label{finale}

Let $S=\{0,s_1,\ldots,s_n,\rightarrow\}$ be a numerical semigroup and $S^c=\{s_1^c,\ldots,s_t^c\}$ be the set of its gaps. By the definition of a numerical semigroup, the sum of any two elements $s_i,s_j\in S$ cannot belong to $S^c$. In other words, each element of $S^c$ cannot be written as a sum of two or more elements from $S$. Therefore, the set of gaps $S^c$ of a numerical semigroup corresponds to an unrefinable partition $\l$. More precisely, the following correspondences hold:
\[
\begin{aligned}
    & S^c=\{s_1^c,\ldots,s_t^c\} &\rightarrow \quad &\l=(\l_1,\ldots,\l_t)\\
    & S=\{0,s_1,\ldots,s_t+1,\rightarrow\} &\rightarrow \quad &\{0\}\cup\mathcal{M}_{\l}\cup\{\l_t+1,\rightarrow\}\\
    & G(S) &\rightarrow \quad &len(\l)\\
    & F(S) &\rightarrow \quad &\l_t\\
    &M(S) &\rightarrow \quad &\mathrm{mex}(\l)=\mu_1
\end{aligned}
\]
An additional correspondence can also be identified: the Ap\'ery set with respect to $s_1$, $Ap(S,s_1)$, coincides with the vector of forbidden elements of the unrefinable partition $\Vec{p}_{\l}$, except for the first position. In the Ap\'ery set, the first element is $0$, while in the vector $\Vec{p}_{\l}$ the first component is $ks_1$ for some $k>1$. 

\begin{example}
    Let $S=\{0,4,7,8,10,11,12,14,\rightarrow\}$ be a numerical semigroup. Its complement is $S^c=\{1,2,3,5,6,9,13\}$, which determines the unrefinable partition $\l=(1,2,3,5,6,9,13)$. Observe that $\mathcal{M}_{\l}=\{4,7,8,10,11,12\}$ coincides with $S^c\setminus(\{0\}\cup\{14,\rightarrow\})$. Furthermore, the number of gaps $G(S)=7=len(\l)$, the Frobenius number $F(S)=13=\l_7$, and $M(S)=4=\mathrm{mex}(\l)$. The Ap\'ery set of $S$ with respect to $4$ is $\{0,17,10,7\}$, while the vector $\Vec{p}_{\l}=\{8,17,10,7\}$.  
\end{example}

As a direct consequence of these correspondences, the following result is obtained.

\begin{proposition}\label{nscontu}
    Let $NS$ be the set of numerical semigroups and let $\mathcal{U}$ denote the set of unrefinable partitions into distinct parts. Then,
    \[
    NS\subset\mathcal{U}.
    \]
\end{proposition}

The converse inclusion does not hold. If $\l\in\mathcal{U}$, then, by definition, no part $\l_i\in\l$ can be expressed as the sum of two equal missing parts. For example, if $\mu_1\in\mathcal{M}_{\l}$, then $2\mu_1$ might be in $\l$, that contradicts the definition of a numerical semigroup, according to which every positive multiple of an element in the semigroup must also belong to it. Thus, any unrefinable partition $\l$, and more generally, any partition into distinct parts, defines a numerical set $S_{\l}$ such that $S_{\l}^c=\l$.\\

The following example shows that $S_{\l}$ is not necessarily a numerical semigroup.  

\begin{example}\label{ex19}
    The partition $\l=(1,2,5,6,8)$ is unrefinable, while the corresponding set $S_{\l}=\{0,3,4,7,9,\rightarrow\}$ is not a numerical semigroup, since $6=3+3$ and $8=4+4$ are not in $S_{\l}$.
\end{example}

By the characterisation of hooksets of numerical semigroups (Proposition \ref{hookset}), an analogous statement can be established for unrefinable partitions.

\begin{lemma}\label{condsemi}
    Let $\l=(\l_1,\ldots,\l_t,\rightarrow)$ be an unrefinable partition corresponding to the Young tableau $Y_{S_{\l}}$, where $S_{\l}$ is the numerical set associated with $\l$. Then
    \begin{enumerate}
        \item The hook length of the cell in the first column and $i$-th row is $\l_i$;
        \item For each $2\leq i\leq \#\mathcal{M}_{\l}$ the hook length of the top cell of the $i$-th column of $Y_{S_{\l}}$ is equal to $\l_t-\mu_{i-1}$;
        \item $\l$ is an unrefinable partition if and only every length of the hook of the cells of       
        $Y_{S_{\l}}$
        \begin{enumerate}
        \item is contained in the first column $Y_{S_{\l}}$;\\
        or
        \item the length of the hook of the cell in the first column and the same row is its double.    
        \end{enumerate}
    \end{enumerate}
\end{lemma}

\begin{proof}
    The proofs of item $(1)$ and item $(2)$ are the same of the proofs of item $(1)$ and item $(2)$ in Proposition \ref{hookset}, and it is true in general for all partitions into distinct parts.\\
    Let $\l$ be an unrefinable partitions. If $\l$ corresponds to a numerical semigroup $S_{\l}$ then, by Proposition \ref{hookset} item $(3)$, all the cells are marked with numbers appearing in the first column, so the statement $(a)$ is proved. If $S_{\l}\notin NS$, then there exists $\mu_j\in\mathcal{M}_{\l}$ such that $k\mu_j\in\l$. In particular $k=2$, otherwise the partition is refinable. By the fact that $2\mu_j\in\l$ there exists one row whose its first cell is marked by $2\mu_j$ (item $(1)$) and all the other cells in the row are such that $2\mu_j-s_i$, with $s_i\in S_{\l}$ and $s_i<\mu_j$. So there is a cell marked by $2\mu_j-\mu_j=\mu_j$. If neither the condition $(a)$ nor the condition $(b)$ are satisfied, then there exists a cell marked by $x$ such that it does not appear in the first column, so, by item $(1)$, $x\notin\l$. Moreover, let $z$ be the first element in the same row of $x$. Then $x=z-s_i$ for some $s_i\in S_{\l}$. Since $z\neq2x$ and $z=x+s_i$, a contradiction arises because $\l$ is refinable.\\
    Conversely, if only condition $(a)$ is satisfied, then $S_{\l}$ is a numerical semigroup and $\l$ is unrefinable. If conditions $(a)$ or $(b)$ are verified, then there exists a cell not in the first column signed by $\mu_j$  and the first element of its row is $2\mu_j$, therefore $2\mu_j\in\l$ by item $(1)$, which cannot be replaced by the sum $\mu_j+\mu_j$, so the partition is unrefinable.
\end{proof} 

\begin{example}
Consider the unrefinable partition $\l=(1,2,5,6,8)$, as in Example \ref{ex19}. Writing the hook length of each cell yields

$$
    \begin{tikzpicture}
        \draw (0,0)--(0.5,0)--(0.5,1)--(1.5,1)--(1.5,2)--(2,2)--(2,2.5)--(0,2.5)--(0,0);
        \draw (0.5,2.5)--(0.5,1);
        \draw (1,2.5)--(1,1);
        \draw (1.5,2.5)--(1.5,2);
        \draw (0,0.5)--(0.5,0.5);
        \draw (0,1)--(0.5,1);
        \draw (0,1.5)--(1.5,1.5);
        \draw (0,2)--(1.5,2);
        \node at (0.25,0.25) {$1$};
        \node at (0.25,0.75) {$2$};
        \node at (0.25,1.25) {$5$};
        \node at (0.25,1.75) {$6$};
        \node at (0.25,2.25) {$8$};

        \node at (0.75,1.25) {$2$};
        \node at (0.75,1.75) {$3$};
        \node at (0.75,2.25) {$5$};
        
        \node at (1.25,1.25) {$1$};
        \node at (1.25,1.75) {$2$};
        \node at (1.25,2.25) {$4$};
        
        \node at (1.75,2.25) {$1$};
        \end{tikzpicture}
$$
Observe that the condition $(a)$ and $(b)$ of Lemma \ref{condsemi} are satisfied, and hence $\l$ is unrefinable.
\end{example}

A subset of unrefinable partitions, that will be useful for establishing relations with numerical semigroups, is now introduced.

\begin{definition}
    Let $\l=(\l_1,\ldots,\l_t)$ be an unrefinable partition, and denote by $\mathcal{U}(\l_t)$ the set of the unrefinable partitions whose maximal part is equal to $\l_t$. Consider the subset of $\mathcal{U}(\l_t)$ consisting of all partitions with the maximal number of missing parts
    \[
    \mathcal{\Bar{U}}(\l_t)=\left\{\l\in\mathcal{U}(\l_t)\mid \#\mathcal{M}_{\l}=\left\lfloor\frac{\l_t}{2}\right\rfloor\right\}
    \]
\end{definition}

Recall that $\l_t = F(S)$. Since the following lemma concerns numerical semigroups, it is more convenient to state it with $\l_t$ as a generic integer.

\begin{lemma} 
    Let $NS(k)$ be the set of numerical semigroups $S$ such that $F(S)=k$. Then 
    \[
    \#NS(k)<\#\mathcal{U}(k).
    \]
\end{lemma}

\begin{proof}
    This is a direct consequence of Proposition \ref{nscontu}.
\end{proof}

Some structural properties of partitions in $\Use$ are now considered.

\begin{lemma}\label{buchispecchio}
    If $\l\in\Use$ and $x\neq\frac{\l_t}{2}$, then $x\in\l$ if and only if $\l_t-x\in\mathcal{M}_{\l}$.
\end{lemma}
\begin{proof}
    By the definition of unrefinable partitions, if $\l_t-x\notin\l$ then $x$ must be a part of $\l$. Conversely, suppose $x\in\l$ and $\l_t-x\in\l$. Since the number of missing parts is $\#\mathcal{M}_{\l}=\left\lfloor\frac{\l_t}{2}\right\rfloor$, then there must exist $y\notin\l$ such that $\l_t-y\notin\l$, which is a contradiction. 
\end{proof}

\begin{lemma}\label{antisymm}
    If $\l=(\l_1,\ldots,\l_t)\in\Use$, then $\frac{\l_t}{2}\notin\l$.
\end{lemma}

\begin{proof}
    If $\l_t$ is an odd number, then the element $\frac{\l_t}{2}$ is not an integer, which means that it cannot appear in $\l$.\\
    Otherwise, if $\l_t$ is even, then the interval $[1,\l_t-1]$ contains exactly $\frac{\l_t}{2}$ missing parts. Let $x\in[1,\frac{\l_t}{2}-1]$, by Lemma \ref{buchispecchio}, $x\notin\l$ implies $\l_t-x\in\l$, and conversely. Thus, in $[1,\frac{\l_t}{2}-1]\cup[\frac{\l_t}{2}+1,\l_t-1]$ there are at most $\frac{\l_t}{2}-1$ missing parts, hence $\frac{\l_t}{2}$ must be missing as well.
\end{proof}

\begin{lemma}\label{multiplo3}
    Let $\l\in\Use$ and let $\l_t$ be an odd integer such that $\l_t\neq3\mu_i$ for all $\mu_i\in\mathcal{M}_{\l}$. Then the numerical set $S_{\l}$ associated to $\l$ is a numerical semigroup.
\end{lemma}
\begin{proof}
    Suppose $S_{\l}$ is not a numerical semigroup. Then there exists $\mu_i\in\mathcal{M}_{\l}$ such that $k\mu_i\in\l$. Since $2\mu_i\notin\l$ implies that no multiples of $\mu_i$ belong to $\l$, it suffices to consider $k=2$. So $2\mu_i\in\l$. By Lemma \ref{buchispecchio}, this yields $\l_t-2\mu_i\notin\l$.\\
    If $\l_t-2\mu_i\neq\mu_i$, then $\l_t-2\mu_i+\mu_i=\l_t-\mu_i\notin\l$, contradicting Lemma \ref{buchispecchio}.Instead, if $\l_t-2\mu_i=\mu_i$, then $\l_t=3\mu_i$, which leads to another contradiction.
\end{proof}

\begin{example}
    Let $\l=(1,2,3,4,7,9,10,15)$ and $\mu=(1,2,3,5,7,9,11,15)$. Observe that $\#\mathcal{M}_{\l}=\#\mathcal{M}_{\mu}=7=\left\lfloor\frac{15}{2}\right\rfloor$, and hence $\l,\mu\in\mathcal{\Bar{U}}(15)$. However, $S_{\l}=\{0,5,6,8,11,12,13,14,16,\rightarrow\}$, since $10,15\notin S_{\l}$, is not a numerical semigroup, while $S_{\mu}=\{0,4,6,8,10,12,13,14,16,\rightarrow\}$ is.
\end{example}

\begin{corollary}
    Let $\l\in\Use$, such that $\l_t$ is an odd integer coprime with $3$, then the numerical set associated $S_{\l}$ is a numerical semigroup.
\end{corollary}

This result follows directly from \Cref{multiplo3}, but it is possible to state even a stronger version.

\begin{corollary}\label{ltprimo}
    Let $\l\in\Use$, such that $\l_t>3$ is a prime number, then the numerical set associated $S_{\l}$ is a numerical semigroup.    
\end{corollary}

\begin{remark}
    Note that the case $\l_t=3$ is excluded. That is because if we consider the partition $\l=(2,3)\in\Use$ the corresponding numerical set $S_{\l}$ is not a numerical semigroup
\end{remark}

\begin{theorem}\label{Ulm=SNS}
    Let $\l_t>3$ be a prime number. Then we have
    \[
    \#\Use=\#\left\{S\in SNS\mid F(S)=\l_t\right\}.
    \]
\end{theorem}

\begin{proof}
    The claim follows from \Cref{ltprimo} and Definition \ref{SNS}.
\end{proof}

A decomposition of $\mathcal{U}(\l_t)$ via the minimal excluded part  is now introduced.\\
Let 
$$
\begin{aligned}    
&\mathcal{U}(\l_t,\mu_1)=\{\l\in\mathcal{U}(\l_t)\mid \mu_1=\mathrm{mex}(\l)\},\\
&\mathcal{\Bar{U}}(\l_t,\mu_1)=\left\{\l\in\mathcal{\Bar{U}}(\l_t)\mid\mu_1=\mathrm{mex}(\l)\right\}
\end{aligned}
$$
Given $\l\in\mathcal{\Bar{U}}(\l_t)$ with $\l_t$ prime number and $\Vec{p}_{\l}$ its vector of forbidden integers  (see Definition \ref{plambda}), Theorem \ref{Ulm=SNS} ensures that $S_{\l}$ is a numerical semigroup. Then $(\Vec{p}_{\l})_1=2\mu_1$ and the other components are the minimal integers not in $\l$, congruent to $i$ $(\mod{\mu_1})$ for $2\leq i\leq\mu_1$.\\
If there exists $1\leq i\leq\mu_1$ such that $(\Vec{p}_{\l})_i\neq(\Vec{p}_{\l})_j+(\Vec{p}_{\l})_k$ for $1\leq j<k\leq \mu_1$ such that $j,k\neq i$, a new unrefinable partition $\l^*\in\mathcal{U}(\l_t,\mu_1)$ such that $\l^*=\l\cup\left\{(\Vec{p}_{\l})_i\right\}$ may be obtained. The vector $\Vec{p}_{\l^*}$ changes from $\Vec{p}_{\l}$ only in the position $i$, i.e., $(\vec{p}_{\l^*})_i=(\vec{p}_{\l})_i+\mu_1$.

\begin{example}
    Let $\l_t=13$ and $\mu_1=3$. Then $\mathcal{\Bar{U}}(\l_t,\mu_1)=\{\l\}$, where $\l=(1,2,4,5,7,10,13)$ and $\Vec{p}_{\l}=(6,16,8)$. Since $13\equiv 1\pmod{3}$, all the lower integers in the same congruence class modulo $3$ are in $\l$.
   The integers in the other congruence classes modulo $\mu_1$ can be added to $\l$ while preserving unrefinability. This procedure produces new unrefinable partitions (see \cref{reticolo}).
\end{example}
\begin{figure}    
    $$
    \begin{tikzcd}[scale=2.5]
    {\{8,11\}} \arrow[r]      & {\{6,8,11\}} \arrow[r]        & {\{6,8,9,11\}} \arrow[r]        & {\{6,8,9,11,12\}}        \\
    \\
    \{8\} \arrow[r] \arrow[uu] & {\{6,8\}} \arrow[uu] \arrow[r] & {\{6,8,9\}} \arrow[r] \arrow[uu] & {\{6,8,9,12\}} \arrow[uu] \\
    \\
    \{0\} \arrow[uu] \arrow[r] & \{6\} \arrow[uu] \arrow[r]     & {\{6,9\}} \arrow[uu] \arrow[r]   & {\{6,9,12\}} \arrow[uu]  
    \end{tikzcd}
    $$
    \caption{Given $\l=\{1,2,4,5,7,10,13\}$, any directed path from $\{0\}$ to $\{6,8,9,11,12\}$ gives a sequence of integers that, added to $\l$, keep it an unrefinable partition. These are all the unrefinable partitions such that $\l_t=13$ and $\mu_1=3$.}
    \label{reticolo}
\end{figure}
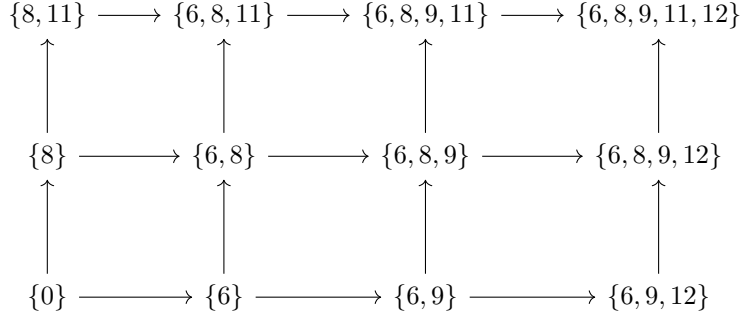

For completeness, an additional result is included. Its proof is straightforward, though its interpretation is less immediate, as it follows from the classification results presented in \cite{ACCL2} and \cite{ACC}. This result may be useful for future research.

\begin{definition}\label{def:maxunref}
Let $N \in \mathbb N$. An unrefinable partition $\l=(\l_1,\ldots,\l_t)$ of  $N$ is called \textbf{maximal} if  
\[
\l_t = \max_{(\eta_1, \eta_2, \dots, \eta_t) \in \mathcal{U}_N}  \eta_t.
\]
Denote by ${\mup{N}}$ the set of the maximal unrefinable partitions of $N$.
\end{definition}

\begin{proposition}
    Let $\widehat{\mathcal{U}}$ be the subset of $\widetilde{\mathcal{U}}$ defined as follows
    \[
    \begin{split}
    \widehat{\mathcal{U}}=& \left\{(1,\ldots,2k-3,4k-6)\mid k\geq4 \right\}\\
    & \bigcup\left\{(1,\ldots,2k-2,4k-5)\mid k\geq4\right\}\bigcup \{\wtp\mid n\geq6\}.
    \end{split}
    \]
    The subset of maximal unrefinable partitions $\Uti\setminus\widehat{\mathcal{U}}$ is contained in the set of unrefinable partitions with maximal number of missing parts $\mathcal{\Bar{U}}$, in other words
    \[
    \left(\Uti\setminus\widehat{\mathcal{U}}\right)\subseteq\Bar{\mathcal{U}}
    \]
\end{proposition}

\begin{proof}
    In \cite{ACCL2} and in \cite{ACC} the authors construct maximal unrefinable partitions. In particular in \cite[Theorem $4.1$]{ACCL2}, the maximal unrefinable partition $\widetilde{\pi}_n=(1,\ldots,n-3,n+1,2n-4)$ is obtained from $\pi_n$ removing three parts and adding two new elements. Then $\#\mathcal{M}_{\widetilde{\pi}_n}=n-3<n-2=\left\lfloor\frac{2n-4}{2}\right\rfloor$. Hence, for every $n\geq6$ the partition $\widetilde{\pi}_n\notin\Bar{\mathcal{U}}$.\\
    In the first part of \cite[Proposition $2.9$]{ACC}, when $n$ is an odd integer, the maximal unrefinable partition $\l=(1,\ldots,n-2,2n-4)$ of $T_{n,3}$ is such that $\#\mathcal{M}_{\l}=n-3<n-2$, hence $\l\notin\Bar{\mathcal{U}}$.\\
    Similarly, in \cite[Proposition $2.12$]{ACC} when $n$ is an even number, the maximal unrefinable partition $\eta=(1,n-2,2n-5)$ for $T_{n,4}$ has $n-4$ missing parts, while $\left\lfloor\frac{2n-5}{2}\right\rfloor=n-3$, so $\eta\notin\Bar{\mathcal{U}}$.\\
    All the others maximal unrefinable partitions obtained in \cite[Theorem $4.1$]{ACCL2}, \cite[Proposition $2.4$, $2.6$, $2.8$, $2.9$, $2.11$, $2.12$]{ACC} have the maximal number of missing parts.
\end{proof}

In conclusion, this work establishes several relationships between unrefinable partitions and numerical semigroups. A future research could be devoted to understand such connection more deeply and finding generating function for $\mathcal{\Bar{U}}(\l_t)$.

\bibliographystyle{amsalpha}
\bibliography{references}

\end{document}